\documentclass[11pt,reqno]{amsart}
\usepackage{amssymb}
\usepackage{graphicx}
\usepackage{xcolor}
\usepackage{enumerate}
\usepackage[pagebackref, colorlinks = true, linkcolor = blue, urlcolor  = blue, citecolor = red]{hyperref}
\usepackage[margin=1in]{geometry}
\usepackage{enumitem}

\newtheorem{theorem}{Theorem}[section]
\newtheorem{definition}[theorem]{Definition}
\newtheorem{proposition}[theorem]{Proposition}
\newtheorem{corollary}[theorem]{Corollary}
\newtheorem{lemma}[theorem]{Lemma}
\newtheorem{remark}[theorem]{Remark}
\newtheorem{example}[theorem]{Example}

\begin{document}

\title{On symmetric decompositions of positive operators}
\author{Maria Anastasia Jivulescu}
\address{MAJ: Department of Mathematics,
Politehnica University of Timi\c soara,
Victoriei Square 2, 300006 Timi\c soara, Romania}
\email{maria.jivulescu@upt.ro}

\author{Ion Nechita}
\address{IN: Zentrum Mathematik, M5, Technische Universit\"at M\"unchen, Boltzmannstrasse 3, 85748 Garching, Germany
and CNRS, Laboratoire de Physique Th\'{e}orique, IRSAMC, Universit\'{e} de Toulouse, UPS, F-31062 Toulouse, France}
\email{nechita@irsamc.ups-tlse.fr}

\author{Pa\c{s}c G\u{a}vru\c{t}a}
\address{PG: Department of Mathematics,
Politehnica University of Timi\c soara,
Victoriei Square 2, 300006 Timi\c soara, Romania}
\email{pgavruta@gmail.com}

\subjclass[2000]{}
\keywords{}

\begin{abstract}
Inspired by some problems in Quantum Information Theory, we present some results concerning decompositions of positive operators acting on finite dimensional Hilbert spaces. We focus on decompositions by families having geometrical symmetry with respect to the Euclidean scalar product and we characterize all such decompositions, comparing our results with the case of SIC--POVMs from Quantum Information Theory. We also generalize some Welch--type inequalities from the literature.
\end{abstract}

\date{\today}

\maketitle

\tableofcontents

\section{Preliminaries}
We study different issues related to the decomposition of a positive operator (i.e.~positive semidefinite matrix) on  a $d$-dimensional complex Hilbert space $\mathcal{H}$, in analogy to the properties of positive-operator valued measures (POVMs) \cite{teiko-book, wol} and frames \cite{ cku, chi, hklw}. We recall that, given a positive operator  $T$ on $\mathcal{H}$ and a family $\mathcal{E}=\{E_1,\ldots, E_N\}$ of positive operators on $\mathcal{H}$, we say that   $\mathcal{E} $ is a \textit{decomposition} of the operator $T$ if we have
$T=\sum_{i=1}^N E_i$. The family $\mathcal{E}$ is called POVM in the case when $T$ is the identity operator on $\mathcal{H}$. POVMs are the most general notion of measurement in quantum theory, and have received a lot of attention in recent years, especially from the Quantum Information Theory community. One of our main goals in this paper is to generalize some of the known results for POVMs to decompositions of arbitrary (positive) self-adjoint operators $T$. 

We focus on the question of decomposing a positive operator $T$ by a symmetric family of positive operators $\mathcal{E}=\{E_1,\dots, E_N\}$. The symmetry of the family $\mathcal{E}$ refers to the geometry of the elements in the Euclidean space of self-adjoint operators: we require the Hilbert-Schmidt scalar products  of all pairs of the elements of the family $\mathcal{E}$ to be the same
$$ \langle E_i,E_j \rangle_{HS} = \operatorname{Tr}[E_i E_j] = a \delta_{ij} + b(1-\delta_{ij}), \qquad \forall  i,j.$$
For a symmetric family, we denote by $a:=\operatorname{Tr}[E_i^2]$ and $b:=\operatorname{Tr}[E_iE_j]$ ($i\neq j$) its symmetry parameters. This type of problem has been  recently asked in the framework of Quantum Information Theory,  for basic connections of this field with operator theory see \cite{gupta}.  A decomposition of the identity by a symmetric family of $N=d^2$ linearly independent operators is called a  \emph{symmetric-informationally complete} POVM (shortly SIC--POVM). Construction of all general SIC--POVMs has been recently  achieved in  the papers \cite{appleby} and \cite{gour}.  With regard to the applications, a particularly difficult problem is the existence of SIC--POVMs whose elements are proportional to rank 1 projections. It is still an open question  if rank 1 SIC--POVMs exist in any dimension; examples have been found for $d=1,\dots, 16,19,24,35,48$ (ananlytical proofs) and $d\leq 67$  (numerically evidence) \cite{scott}. The closeness  of general SIC--POVM to a rank 1 SIC--POVM has been recently quantified \cite{gour} using the parameter $a$ that characterizes the symmetric family. We follow the same path of investigation for decompositions of arbitrary operators $T$ and we give bounds for the symmetry parameter $a$. Our motivation is to achieve a better understanding of the more general situation, with the hope that this will shed some light on the more interesting case of unit rank SIC-POVMs. 

The paper is organized as follows.
In Section \ref{decomp} we gather some relatively straightforward general properties of decompositions, proving that a local decomposition for an injective operator is essentially a global one, and characterizing decompositions of orthogonal projections. In Sections \ref{sec:symmetric-decompositions}  and \ref{4} we focus on symmetric decompositions of general, and then positive operators; these sections contain the main result of the paper, a characterization of all symmetric (positive) decompositions of a given (positive) operator. Finally, Section \ref{5} contains some generalization of weighted Welch--type inequalities.

\bigskip 

\noindent \textit{Acknowledgments.} The work of M.A.J. and P.G. was supported by a grant of the Romanian National Authority for Scientific Research, CNCS-UEFISCDI, project number  PN-II-ID-JRP-RO-FR-2011-2-0007.
I.N.'s research has been supported by a von Humboldt fellowship and by the ANR projects {RMTQIT}  {ANR-12-IS01-0001-01} and {StoQ} {ANR-14-CE25-0003-01}.

\section{Properties of general decompositions}
\label{decomp}
In this paper, we study operators acting on a Hilbert space $\mathcal H$, which will be finite dimensional (with the exception of Proposition \ref{prop:local-decompositions}) and complex (unless otherwise specified). Our focus will be on the existence of \emph{decompositions} of operators as sums of families having some specific symmetry or positivity properties.

\begin{definition}
Let be $\mathcal{E}=\{E_i\}_{i=1}^N$ a family of self-adjoint operators acting on some Hilbert space $\mathcal H$, and $T$ another given operator acting on $\mathcal H$. We say that $\mathcal E$ is a \emph{decomposition} of $T$ if $\sum_{i=1}^N E_i = T$. If the operators $E_i$ are positive (i.e.~$E_i$ are positive semidefinite matrices), the family $\mathcal E$ is called \emph{positive}.

We say that $\mathcal{E}$ is a \emph{local decomposition} of $T$ if for any $x\in\mathcal{H}$, $x\neq 0$, there is a complex number $\alpha(x)\neq 0$ such that
\begin{equation}\label{alfa}
\sum_{i=1}^N E_i x=\alpha(x)Tx.
\end{equation}
\end{definition}

\begin{proposition}
\label{prop:local-decompositions}
Let $\mathcal H$ be any (possibly infinite dimensional) complex Hilbert space, and $T$ an operator acting on $\mathcal H$. If $T$ is injective and $\mathcal{E}$ is a local decomposition of $T$, then there exists $\beta \in \mathbb C$ such that $\mathcal E$ is a decomposition of $\beta T$.

\end{proposition}
\begin{proof} 
We first consider $x,y$ linearly independent and we show that $\alpha(x)=\alpha(y)$. By linearity, we have
$\alpha(x+y)T(x+y)=\alpha(x)Tx+\alpha(y)Ty$ and since $T$ is injective, we get $\alpha(x+y)(x+y)=\alpha(x)x+\alpha(y)y$. Using the linear independence of $x$ and $y$, together with
$[\alpha(x+y)-\alpha(x)]x+[\alpha(x+y)-\alpha(y)]y=0$, we conclude that $\alpha(x)=\alpha(y) = \alpha(x+y)$.

Let now $x,y\neq 0, y=\lambda x,\lambda\neq 0$. 
We have $\sum E_i x=\alpha(x)Tx$, $\sum E_i y=\alpha(y)Ty\Leftrightarrow \sum E_i(\lambda x)=\alpha(\lambda x)T(\lambda x)$, hence $\sum E_i x=\alpha(\lambda x)Tx$, so $\alpha(x)Tx=\alpha(\lambda x)Tx \Rightarrow \alpha(\lambda x)=\alpha(x)$, as $T$ is injective. This shows that the function $\alpha$ is constant, finishing the proof.
\end{proof}
\begin{remark}
If the operator $T$ and the decomposition $\mathcal E$ are positive, then the scalar $\beta$ is non-negative, $\beta \geq 0$.
\end{remark}
\begin{remark}
In the particular case when the Hilbert space $\mathcal H$ is finite dimensional, we can prove the above result in the following way. The operator $T$ was assumed injective, hence it is invertible, and we have
$$\forall x \neq 0, \qquad T^{-1}\left( \sum_{i=1}^N E_i \right) x = \alpha(x) x.$$
Hence, every non-zero vector $x \in \mathcal H$ is an eigenvector of $T^{-1}\sum_i E_i$; this cannot happen unless $T^{-1}\sum_i E_i$ has just one eigenvalue (eigenspace), i.e.~it is a multiple of the identity. 
\end{remark}

Let $T$ be a self-adjoint operator. For any given scalar weights $(t_1, \ldots, t_N)$ such that $\sum_{i=1}^N t_i=1$, it is clear that $\{t_iT\}$ is a decomposition of $T$ (if $T$ is positive and one assumes $t_i \geq 0$ for all $i$, the decomposition is positive). Such a decomposition is called \emph{degenerate}. We prove that if $\operatorname{rk}[T]=1$, then $T$ has only this type of decomposition.

\begin{proposition}
Let $T$ be a positive operator acting on a finite dimensional complex Hilbert space $\mathcal H$, and $\mathcal{E}=\{E_i\}_{i=1}^N$ a decomposition of $T$. If $\mathrm{rk}[T]=1$, then $E_i=t_iT$, $0\leq t_i\leq 1$ for $i\in\{1,2,\ldots, N\}$ and $\sum_{i=1}^N t_i=1$.
\end{proposition}

\begin{proof}
This is an easy consequence of the fact that the extremal rays of the positive semidefinite cone are unit rank projections. More precisely, from $T=\sum_{i=1}^N E_i$, it follows that $T\geq E_i, \forall i$. We apply Proposition (1.63), from \cite{teiko-book}. It follows that $E_i=t_iT$, for $0\leq t_i\leq 1$, $i=1,2\ldots N$, hence $T=(\sum_{i=1}^N t_i)T$. Since $\operatorname{rk}[T]=1$, it follows that there is $x\in\mathcal{H}$ such that $Tx\neq 0$. Hence, $\sum_{i=1}^N t_i=1$. 
\end{proof}

We characterize next decompositions of self-adjoint projections. 

\begin{proposition}\label{proj}
Let $\mathcal{E}$ be a positive decomposition of a positive operator $T$. Then, the following are equivalent:
\begin{enumerate}[label=\roman*)]
\item $T$ is projection
\item $TE_i=E_iT=E_i$, for all $i$.
\end{enumerate} 
\end{proposition}

\begin{proof}
$i)\Rightarrow ii)$. We know that $T\geq E_i\geq 0$. From \cite[Proposition 1.46]{teiko-book}, it follows that  $T E_i=E_i T=E_i$.

$ii)\Rightarrow i)$. We have 
$T^2=T\sum_{i=1}^N E_i=\sum_{i=1}^N T E_i = \sum_{i=1}^N E_i=T$.
\end{proof}

\section{Decomposition of a positive operator by a symmetric family}
\label{sec:symmetric-decompositions}
In this section we are going to study decompositions of operators by families having the following form of symmetry. 

\begin{definition}
A family $\mathcal{E}=\{E_i\}_{i=1}^N$ of self-adjoint operators acting on a  finite dimensional Hilbert space $\mathcal H$ is called \emph{symmetric} if, for all $i,j\in\{1,2,\ldots,N\}$,
\begin{equation}\label{eq:symmetric-family}
\mathrm{Tr}[E_iE_j]=a\delta_{ij}+b(1-\delta_{ij}).
\end{equation}
The scalars $a,b$ are called the parameters of $\mathcal E$. 
\end{definition}

Most of the results below hold for general decompositions of self-adjoint operators. However, when the operator $T$ is positive and we require that the decomposition should also be positive, some additional structure emerges, and we emphasize this case in the respective results. 

\begin{proposition} \label{prop:relation-a-b-T}
Given a self-adjoint operator $T$, let $t_2 := \operatorname{Tr} [T^2]$. Consider $\mathcal E$ a symmetric family of self-adjoint operators as in \eqref{eq:symmetric-family}, having parameters $a$ and $b$. Then, $\mathcal{E}$ is symmetric decomposition of $T$ if and only if the following relations hold
\begin{equation}\label{b}
b=\frac{t_2-Na}{N(N-1)}
\end{equation}
and
\begin{equation}\label{b1}
\mathrm{Tr}[E_iT]=\frac{t_2}{N}, \qquad i=1,2,\ldots N
\end{equation}
\end{proposition}

\begin{proof}
We assume first that $\mathcal{E}$ is a symmetric decomposition of $T$. We have that $T^2=\sum_{i,j=1}^N E_i E_j$, hence $$t_2 = \mathrm{Tr} [T^2]=\sum_{i,j=1}^N \mathrm{Tr}[E_iE_j]=Na+N(N-1)b.$$ It follows that $$b=\frac{t_2-Na}{N(N-1)}.$$
For $i\in\{1,2,\ldots N\}$, we have  $$\mathrm{Tr}[E_iT]=\mathrm{Tr}[E_i\sum_{j=1}^N E_j]=\mathrm{Tr}[E_i^2]+\sum_{\substack{j=1 \\ j\neq i}}^N\mathrm{Tr}[E_iE_j]=a+(N-1)b.$$
Using \eqref{b} we get that 
$$\mathrm{Tr}[E_iT]=a+\frac{t_2-Na}{N}=\frac{t_2}{N}.$$
Conversely, if \eqref{b} and \eqref{b1} hold, then we have
\begin{align*}
\left\| \sum_{i=1}^N E_i-T \right\|^2_{HS} &=\mathrm{Tr}[(\sum_{i=1}^N E_i-T)(\sum_{j=1}^NE_j-T)] \\
&= \mathrm{Tr}[\sum_{i,j=1}^N E_i E_j ] - \sum_{i=1}^N \mathrm{Tr}[E_iT]-\sum_{j=1}^N\mathrm{Tr}[TE_j] + \mathrm{Tr}[T^2] \\
&=Na+N(N-1)b-N\frac{t_2}{N}-N\frac{t_2}{N}+t_2=0.
\end{align*}
\end{proof}

\begin{proposition}\label{prop:a-lower-bound}
Let $\mathcal{E}$ be a symmetric decomposition of  a self-adjoint operator $T$. Then, the parameter $a$ of the family $\mathcal E$ satisfies $ a \geq t_2/N^2$. If $a>t_2 / N^2$ and $T \neq 0$, then the set $\{E_1,\ldots E_N\}$ is linear independent.  If $a=t_2 / N^2$, then the decomposition $\mathcal E$ is degenerate: $E_i=t_iT,$ for all $i$ and $\sum_{i=1}^N t_i=1$.
\end{proposition}

\begin{proof}
To establish the bound, we use the Cauchy--Schwarz inequality
$$(\mathrm{Tr}[E_iT])^2\leq \mathrm{Tr} [E_i^2]\cdot \mathrm{Tr} [T^2]$$
 and Proposition \ref{prop:relation-a-b-T} to get $(t_2/N)^2\leq a t_2$, so $t_2 / N^2\leq a$.
 
We suppose now that $a>t_2/N^2$. From Proposition \ref{prop:relation-a-b-T} it follows that this inequality is equivalent to the inequality $a>b$. Consider scalars $\lambda_i$ such that $\sum_{i=1}^N\lambda_i E_i = 0$. We have 
\begin{align*}
0 &= \left\|\sum_{i=1}^N\lambda_i E_i \right\|_{HS}^2=\mathrm{Tr}\left[ \left(\sum_{i=1}^N\bar{\lambda_i}E_i \right) \left( \sum_{j=1}^N\lambda_jE_j \right) \right] \\
&= \sum_{i,j=1}^N \bar{\lambda_i}\lambda_j \mathrm{Tr}(E_iE_j)=a\sum_{i=1}^N|{\lambda}_i|^2+b\sum_{\substack{i,j=1 \\ i\neq j}}^N\bar{{\lambda}_i}{\lambda}_j\\
&= (a-b)\sum_{i=1}^N|{\lambda}_i|^2+b\left|\sum_{i=1}^N{\lambda}_i \right|^2.
\end{align*}
If $b$ is non-negative, all the terms in the sum above are zero, so $\lambda_1 = \cdots = \lambda_N = 0$, showing that the family of operators operators $\mathcal E = \{E_i\}_{i=1}^N$ is linearly independent. On the other hand, if $b<0$, we have, using again Cauchy-Schwarz,
$${\frac{a-b}{-b} = \frac{\left|\sum_{i=1}^N{\lambda}_i \right|^2}{\sum_{i=1}^N|{\lambda}_i|^2}} \leq N,$$
and thus $b \leq -a/(N-1)$. However, this contradicts \eqref{b}: $T \neq 0$ and thus
$$b=\frac{t_2-Na}{N(N-1)} > -\frac{a}{N-1}.$$

Finally, if $a=t_2 / N^2$, using the equality case in the Cauchy-Schwarz inequality, we get $E_i=t_iT$ for some scalars $t_i$.  We have thus $T=(\sum_{i=1}^N t_i)T$, and we can choose the $t_i$ such that $\sum_{i=1}^N t_i=1$, finishing the proof.
\end{proof}

As a corollary of this result, we obtain a generalisation of \cite[Proposition 4.2]{cahill}.

\begin{corollary}
Let $\mathcal E = \{E_i\}_{i=1}^N$ be a symmetric family of self-adjoint operators acting on a Hilbert space $\mathcal H$, and let $d=\dim \mathcal H$. If the operators in $\mathcal E$ are pairwise distinct, then $N \leq d^2$. Moreover, if $\mathcal H$ is a real Hilbert space, then $N \leq d(d+1)/2$.
\end{corollary}
\begin{proof}
Since the operators in $\mathcal E$ are pairwise distinct, from the Cauchy-Schwarz inequality it follows that $b < a$, where $a,b$ are the parameters of the symmetric family $\mathcal E$. In Proposition \ref{prop:a-lower-bound}, it has been shown that in this case, the operators $\{E_1, \ldots, E_N\}$ are linearly independent. Thus, $N$ must be at most the dimension of the space of self-adjoint operators on $\mathcal H$, proving the claim. 
\end{proof}

One can also upper bound the parameter $a$ of a symmetric decomposition in the case where the elements of the decomposition are positive operators (i.e.~positive semidefinite matrices). 

\begin{proposition}\label{prop:a-upper-bound}
Let $\mathcal{E}$ be a positive symmetric decomposition of  a positive operator $T$. Then, the parameter $a$ of the family $\mathcal E$ satisfies $ a \leq t_2/N$.
\end{proposition}
\begin{proof}
For any $i$, we have $T \geq E_i$ so $\operatorname{Tr}[E_i(T-E_i)] \geq 0$, hence, using Proposition \ref{prop:relation-a-b-T}, $a\leq \mathrm{Tr}[E_iT]=t_2 / N$. 
\end{proof}

\begin{proposition} If $T$ is a projection, then 

\begin{equation}
\frac{t_2}{N^2}\leq a\leq \frac{t_2^2}{N^2}
\end{equation}
The upper bound is saturated if and only if $E_i$ is of rank one, for $i=1,2,\dots N$.
\end{proposition}
\begin{proof} We have $a=\mathrm{Tr}[E_i^2]\leq (\mathrm{Tr} E_i)^2=t_2^2 / N^2$, since if $T$ is projection, then by Proposition \ref{proj} $E_iT=TE_i=E_i$. We have equality iff $\mathrm{rk} [E_i]=1,$ for all $1 \leq i \leq N$.
\end{proof}

Assuming that the operators $E_i$ are positive and \emph{invertible}, one can derive a different upper bound for the parameter $a$ of the decomposition than the one obtained in Proposition \ref{prop:a-upper-bound}. Recall that the condition number of an invertible operator $A$ is defined as 
$$\kappa(A) := \|A\| \cdot \|A^{-1}\|.$$
In the case $A$ is a strictly positive operator, we have
$$\kappa(A) = \frac{\lambda_{\max}(A)}{\lambda_{\min}(A)}.$$

\begin{proposition} If $T$ is positive operator which has the decomposition $T=\sum_{i=1}^N E_i$, where $E_i,  i=1,\ldots N$ are strictly positive operators, then we have
\begin{equation}
\frac{t_2}{N^2}\leq a\leq M \frac{t_2}{N^2}
\end{equation}
where 
\begin{equation}
M=\min_{1\leq i\leq N} \frac{1}{4} \left( \sqrt{\kappa(T) \kappa(E_i)} + \frac{1}{\sqrt{\kappa(T) \kappa(E_i)} }\right)^2.
\end{equation}
\end{proposition}
\begin{proof}
We use  Proposition \ref{prop:a-lower-bound} and the converse Cauchy-Schwarz inequality \cite[Corollary 1.4]{niculescu}.
\end{proof}

We describe next all the possible symmetric decompositions of a given self-adjoint operator $T$ acting on a finite dimensional Hilbert space. The result is a generalization of \cite[Theorem 3]{gour} to general operators $T$. Moreover, we give necessary conditions for the existence of a positive decomposition of $T$, in the case where $T$ is a positive operator (one can assume actually $T$ to be invertible in this case, since any positive decomposition of $T$ is supported on the orthogonal of the kernel of $T$).

\begin{proposition}\label{prop:general-symmetric-decompositions}

Let $T$ be a self-adjoint operator acting on $\mathcal H \simeq \mathbb C^d$, and consider an operator subspace $\mathcal F$ orthogonal to $\mathbb C T$, of dimension $N-1$. Then, the set of symmetric decompositions $\mathcal E = \{E_i\}_{i=1}^N$ of $T$ with support in $\mathbb CT \oplus \mathcal F$ is in bijection with $N$-tuples $(x, F_1, F_2, \ldots, F_{N-1})$, where $x$ is a non-negative number and  $F_1, \ldots, F_{N-1}$ is an orthonormal basis of $\mathcal F$. The bijection can be described as follows:  put $F := \sum_{i=1}^{N-1} F_i$, and define 
\begin{align}
\label{eq:R-i} R_i &= \frac{F}{\sqrt{N-1}(\sqrt N +1)} - \frac{\sqrt N}{\sqrt{N-1}} F_i, \qquad i=1, \ldots, N-1\\
\label{eq:R-N} R_N &= \frac{F}{\sqrt{N-1}}.
\end{align}
Then, for any non-negative real $x$, the operators
\begin{equation} \label{eq:E-i}
E_i := \frac{1}{N} T + x R_i, \qquad i=1,\ldots, N
\end{equation}
define a symmetric decomposition $T = \sum_{i=1}^N E_i$ of $T$, with parameters 
\begin{align}
\label{eq:a-general-decomposition} a &= \frac{t_2}{N^2} + x^2\\
b &= \frac{t_2}{N^2} - \frac{x^2}{N-1},
\end{align}
where $t_2 = \operatorname{Tr}[T^2] = \|T\|_\mathrm{HS}^2$. 
Reciprocally, all symmetric decompositions of $T$ can be obtained as described above. 

Assume now that $T$ is positive definite matrix and let $-\mu_i$ denote the smallest eigenvalue of $R_i$; since $\operatorname{Tr}[R_iT]=0$ for all $i$, we have $\mu_i>0$. If, moreover, 
\begin{equation}\label{eq:interval-x-PSD}
x \leq \frac{\tau}{N \max_{i=1}^N \mu_i},
\end{equation}
where $\tau = \lambda_{\min}(T)$, the operators $\{E_i\}_{i=1}^N$ are positive semidefinite. 
\end{proposition}
\begin{proof}
The proof follows closely \cite[Theorem 3]{gour}, with a different normalization of the operators $R_i$. Let us show first the relation between the angles among the $F_i$'s and the angles among the $E_j$'s. Starting from an orthonormal family $F_1, \ldots, F_{N-1}$, by direct computation, and using facts such as $\langle F_i, F \rangle = 1$, $\langle F, F \rangle = N-1$, $\langle R_i, T \rangle = 0$, the symmetry of the family $\{E_i\}_{i=1}^N$ follows, namely
\begin{equation}\label{eq:symmetry-relation-proof}
\langle E_i , E_j \rangle = a \delta_{ij} + b(1-\delta_{ij}).
\end{equation}
The decomposition property $\sum_i E_i = T$ follows from the fact that $\sum_i R_i = 0$, which can be shown directly from \eqref{eq:R-i} and \eqref{eq:R-N}. 

Reciprocally, one has $x = \|E_i - T/N\|_2$ and we can write the operators $F_i$ in terms of the $E_j$ working back the equations \eqref{eq:R-i},\eqref{eq:R-N},\eqref{eq:E-i}.
The orthonormality of the $F_i$'s and the fact that $F_i \perp T$, for all $1\leq i \leq N-1$ follow now from the symmetry relation \eqref{eq:symmetry-relation-proof}.

Let us now discuss the positivity of the operators $E_i$. We have, by standard inequalities,
$$\lambda_{\min}(E_i) \geq \frac{\tau}{N} + x \lambda_{\min}(R_i) = \frac{\tau}{N} - x \mu_i.$$
Hence, if $x \geq 0$ is as in \eqref{eq:interval-x-PSD}, then necessarily $E_i \geq 0$.
\end{proof}

\begin{remark}
The equations \eqref{eq:R-i},\eqref{eq:R-N},\eqref{eq:E-i} relating the operators $E_i$ to the orthonormal basis $F_j$ can be summarized as follows:
$$\forall 1 \leq i \leq N, \qquad E_i = \frac{1}{N}T+x\sum_{j=1}^{N-1} V_{ij}F_j,$$
where $V \in M_{N \times (N-1)}(\mathbb R)$ is the following matrix
$$V =\frac{1}{{\sqrt{N-1}(\sqrt{N}+1) }}\begin{bmatrix}
1-\sqrt N - N & 1 & \cdots  & 1\\
1 & 1-\sqrt N - N & \cdots  & 1\\
\vdots & \vdots & \ddots & \vdots \\
1 & 1 & \cdots & 1-\sqrt N - N\\
\sqrt N + 1 & \sqrt N + 1 & \cdots & \sqrt N + 1
\end{bmatrix}.$$
Note that $V$ is a multiple of an isometry which maps $\mathbb R^{N-1}$ to the orthogonal of $\mathbb R (1, 1, \ldots, 1)^\top$ in $\mathbb R^N$: we have $V^\top V = N/(N-1)I_{N-1}$ and $VV^\top = N/(N-1)(I_N - N^{-1}J_N) = N/(N-1)P_{\neq 1}$, where $P_{\neq 1}$ is the projection on the orthogonal of $\mathbb R(1, \ldots, 1)^\top$ in $\mathbb R^N$.
\end{remark}

\begin{remark}
When $T = I$ and $N = d^2$, we recover \cite[Theorem 3]{gour}.
\end{remark}

\begin{remark}
Note that the geometric parameters $a$ and $b$ of the decomposition $\mathcal E$ depend only on the square $x^2$ of the free parameter $x$; this is related to the fact that if one allows negative values of $x$, the $N$-tuples $(x, F_1, \ldots, F_{N-1})$ and $(-x, -F_1, \ldots, -F_{N-1})$ give the same decomposition of $T$. 
\end{remark}

In order to obtain an upper bound for the parameter $a$ of a symmetric decomposition of a positive operator $T$, we need the following lemma. 

\begin{lemma}\label{lem:optimization-A-B}
Let $B$ be a positive definite operator acting on $\mathcal H \simeq \mathbb C^d$, having eigenvalues $b_1 \geq b_2 \geq \cdots \geq b_d > 0$. The following two optimization problems are equivalent:
\begin{equation*}
\begin{aligned}[c]
(P_1): \quad \qquad \min \quad &\lambda_{\max}(A)\\
\mathrm{subject\,to}\quad  &A = A^* \\
& \|A\|_\mathrm{HS} = 1\\
&A \perp B
\end{aligned}
\qquad\qquad\qquad\qquad
\begin{aligned}[c]
(P_2): \quad \qquad \min \quad & a_1\\
\mathrm{subject\,to}\quad    & a_1 \geq a_2 \geq \cdots \geq a_d \in \mathbb R \\
& \textstyle{\sum_{i=1}^d a_i^2} = 1\\
& \textstyle{\sum_{i=1}^d a_i b_i} \geq 0\\
& \textstyle{\sum_{i=1}^d a_i b_{d+1-i}} \leq 0,
\end{aligned}
\end{equation*}
and have common value 
\begin{equation}\label{eq:optimization-value}
\varphi(B) = \varphi(b) = \left[(\beta/b_d -1)^2 + d-1 \right]^{-1/2},
\end{equation}
where $\beta = \operatorname{Tr}[B] = \sum_{i=1}^d b_i$.
\end{lemma}
\begin{proof}
Let us first show that the programs $(P_{1,2})$ are equivalent, and then solve the easier, scalar version $(P_2)$. To show equivalence, note that the objective function and the Hilbert-Schmidt normalization condition in $(P_1)$ are spectral, i.e.~they depend only on the eigenvalues $a_1 \geq a_2 \geq \cdots \geq a_d$ of $A$. The equivalence of $(P_{1,2})$ follows from the following fact: given to spectra $a^\downarrow = (a_1 \geq \cdots \geq a_d)$ and $b^\downarrow = (b_1 \geq \cdots \geq b_d)$ with $b_d > 0$, there exist a unitary operator $U$ acting on $\mathbb C^d$ such that $U \operatorname{diag}(a_1, \ldots, a_d) U^* \perp \operatorname{diag}(b_1, \ldots, b_d)$ iff 
\begin{equation}\label{eq:a-b-scalar-products}
\langle a^\downarrow, b^\downarrow \rangle \geq 0 \qquad \text{and} \qquad \langle a^\downarrow, b^\uparrow \rangle \leq 0,
\end{equation}
where $b^\uparrow = (b_d \leq \cdots b_1)$; note that the conditions above are precisely those appearing in $(P_2)$. The property above is implied by the following fact (see \cite[Theorem 4.3.53]{hjo}):
$$\{\operatorname{Tr}[U \operatorname{diag}(a_1, \ldots, a_d) U^*  \operatorname{diag}(b_1, \ldots, b_d)] \, : \, U \text{ unitary}\} = [\langle a^\downarrow, b^\uparrow \rangle,\langle a^\downarrow, b^\downarrow \rangle].$$

Let us now solve $(P_2)$. The proof will consist of two steps: we shall show first that an optimal vector $a$ is necessarily of the following form:
\begin{equation}\label{eq:a-tweo-valued}
a = (\underbrace{s, \ldots, s}_{d-1 \text{ times}},-r),
\end{equation}
for some $r,s>0$. We shall then optimize over vectors of this form. 

For the first step, let us consider a feasible vector $a$ which is not as in \eqref{eq:a-tweo-valued}: $a = (a_1 = \cdots = a_m > a_{m+1} \geq \cdots \geq a_d$; here, $m+1 < d$. Since $a$ is at least three-valued, there exist $m<i<j$ such that $a_i > a_j$. Moreover, let us assume that $i$ is the smallest index where $a$ takes the value $a_i$ and $j$ is the largest index where $a$ takes the value $a_j$; we have thus $a_i < a_{i-1}$ and, if $j<d$, $a_j > a_{j+1}$. Let us define the vector $a'$ by
$$a'_k = \begin{cases} 
a_i + \varepsilon & \qquad \text{ if } k=i\\
a_j - \varepsilon & \qquad \text{ if } k=j\\
a_k & \qquad \text{ if } k \neq i,j,
\end{cases}$$
where $\varepsilon > 0$ is the largest such that $a'_1 \geq \cdots \geq a'_d$. In terms of the majorization relation (see \cite[Chapter II]{bha}), we have $a \prec a'$, so the scalar product relations \eqref{eq:a-b-scalar-products} still hold for $a'$. Note however that $a'$ is not feasible, since 
$$\|a'\|_2^2 - \|a\|_2^2 = 2\varepsilon(a_i-a_j) + \varepsilon^2 > 0.$$
We normalize $a'$ by $\|a'\|_2 > 1$: $a'' = a' / \|a'\|_2$. Obviously, $a''$ is feasible and moreover 
$$a''_1 = \frac{a'_1}{\|a'\|_2} = \frac{a_1}{\|a'\|_2}<a_1,$$
and thus $a$ cannot be optimal.

Let us now optimize over two-valued vectors $a$ as in \eqref{eq:a-tweo-valued}. The conditions in $(P_2)$ read, respectively (we put $\beta = \sum_i b_i$)
\begin{align*}
r^2 + (d-1)s^2 &= 1\\
s(\beta - b_d) - r b_d &\geq 0\\
s(\beta - b_1) - r b_1 &\leq 0.
\end{align*}
Thus, $(P_2)$ is equivalent to minimizing $s$ under the constraints 
$$\frac{b_d}{\sqrt{(\beta-b_d)^2 + (d-1)b_d^2}} \leq s \leq \frac{b_1}{\sqrt{(\beta-b_1)^2 + (d-1)b_1^2}},$$
and the conclusion follows. 

\end{proof}

\begin{remark}\label{rem:-lambda-min}
Note that in the optimization problem $(P_1)$ over self-adjoint matrices $A$, one could have replaced the objective function by $-\lambda_{\min}(A)$; this follows from the observation that the feasible set is invariant by sign change. 
\end{remark}

\begin{remark}
If $T = I$, then $\varphi(I) = \varphi(1, \ldots, 1) = 1/\sqrt{d(d-1)}$. Note also that, for arbitrary $B$, $\beta / b_d \geq d$, so $\varphi(B) \geq \varphi(I)$, for all $B > 0$.
\end{remark}

Equation \eqref{eq:interval-x-PSD} from Proposition \ref{prop:general-symmetric-decompositions} gives a sufficient condition for the variable $x$ in order for a decomposition of $T$ to be positive. Note that value in \eqref{eq:interval-x-PSD} might not be tight: larger values of $x$ might yield positive decompositions. We present next a necessary condition the parameter $a$ (and thus $x$) must satisfy in order for a positive symmetric decomposition of $T$ with those parameters to exist.

\begin{proposition}\label{prop:a-upper-bound-T-PSD}
Let $T$ be a positive operator acting on $\mathcal H \simeq \mathbb C^d$. Any positive symmetric decomposition of $T$  has parameter $a$ such that
\begin{equation}\label{eq:bound-a-decomposition}
\frac{t_2}{N^2}\leq a\leq \frac{t_2 + \|T\|_\infty^2 / \varphi^2}{N^2} = \frac{t_2}{N^2} + \left[\frac{ \|T\|_\infty}{\tau}\right]^2 \cdot \frac{(t_1 - \tau)^2 + (d-1)\tau^2}{N^2}, 
\end{equation}
where $t_{1,2} = \operatorname{Tr}[T^{1,2}]$, $\tau = \lambda_{\min}(T)$, $\|T\|_\infty = \lambda_{\max}(T)$, and $\varphi:=\varphi(T)$ was defined in Lemma \ref{lem:optimization-A-B}.
\end{proposition}
\begin{proof}
The lower bound was shown in Proposition \ref{prop:a-lower-bound}. For the upper bound, using \eqref{eq:a-general-decomposition}, we need to upper bound $x$. Since any positive value of $x$ gives a symmetric decomposing family $\mathcal E$, the only constraints on $x$ come from the positivity of the operators $E_i$. Using \eqref{eq:E-i}, we have
$$0 \leq \lambda_{\min}(E_i) \leq \lambda_{\max}(T/N) + \lambda_{\min}(xR_i) = \frac{\|T\|_\infty}{N} + x \lambda_{\min}(R_i).$$
Writing $\varphi$ for the value of the optimization problem in Lemma \ref{lem:optimization-A-B} with $B=T$
(see also Remark \ref{rem:-lambda-min}), we have $-\lambda_{\min}(R_i) \geq \varphi$, so 
\begin{equation}\label{eq:bound-a-ineq}
x \leq \frac{\|T\|_\infty}{N\varphi},
\end{equation}
which, together with \eqref{eq:a-general-decomposition} and \eqref{eq:optimization-value} gives the announced bound. 
\end{proof}
\begin{remark}
In the case $T=I$ and $N = d^2$, the bound \eqref{eq:bound-a-decomposition} reads $d^{-3} \leq a \leq d^{-2}$, which was also found in \cite[Eq.~(2)]{gour}.
\end{remark}

Let us discuss now the simplest case, $d=2$. We have the following result, characterizing the equality cases in the upper bound \eqref{eq:bound-a-decomposition}, when $N=d^2=4$.

\begin{proposition}
In the case $d=2$, $N=4$, consider the general construction of a symmetric decomposition of a positive definite operator $T \in M_2(\mathbb C)$ from Proposition \ref{prop:general-symmetric-decompositions}. The following statements are equivalent:
\begin{enumerate}
\item There exists an orthonormal basis $\{F_1, F_2, F_3\}$ of $(\mathbb C T)^\perp$ such that the decomposition $T = \sum_{i=1}^4 E_i$ saturates the upper bound \eqref{eq:bound-a-decomposition}.
\item The operator $T$ is a multiple of the identity.
\item For any orthonormal basis $\{F_1, F_2, F_3\}$ of $(\mathbb C T)^\perp$, the decomposition $T = \sum_{i=1}^4 E_i$ saturates the upper bound \eqref{eq:bound-a-decomposition}.
\end{enumerate}
\end{proposition}
\begin{proof}
Let us first show $(2) \implies (3)$, assuming $T=I_2$. Start from any orthonormal basis $\{F_1,F_2,F_3\}$ of $(\mathbb CT)^\perp$. 
Since the matrices $R_j$ have unit Schatten $2$-norm and are also traceless, they have eigenvalues $\pm 1/\sqrt{2}$, so the optimal constants $x$ and $a$ from Proposition \ref{prop:general-symmetric-decompositions} read, respectively, $x = \sqrt{2}/4$ and $a=1/4$, which is indeed the upper bound \eqref{eq:bound-a-decomposition}.

We show now $(1) \implies (2)$. Les us consider the inequality \eqref{eq:bound-a-ineq} which leads to the upper bound \eqref{eq:bound-a-decomposition}. Assuming the equality in \eqref{eq:bound-a-ineq} was achieved, i.e.
$$x = \frac{\|T\|_\infty}{N \max_i \mu_i},$$
we get $\mu_1 = \mu_2 = \mu_3 = \mu_4 = \|T\|_\infty/(N x_*)$, where $x_*$ is the optimal value of $x$ needed to achieve \eqref{eq:bound-a-decomposition}. In particular, since the matrices $R_j$ have the same Hilbert-Schmidt norm, they must be isospectral, so the respective positive eigenvalues $\rho_{1,2,3,4}$ of the $R_j$ matrices are also equal. Putting $\delta = \operatorname{Tr} R_j = \rho_j - \mu_j$, we get from \eqref{eq:R-i} and \eqref{eq:R-N} $\delta = 0$, so the $R_j$ are traceless, which implies $T = cI_2$, for some constant $c > 0$. 
\end{proof}

The result above excludes the existence of ``SIC--POVM--like'' decompositions for positive operators $T \neq cI_2$ which would saturate the upper bound for the norm of the operators. On the other hand, for $d=2$ and $T=I_2$, any starting orthonormal basis for the traceless operators produces a SIC--POVM, so starting from Pauli matrices as in \cite[Section 6]{gour} is not necessary in this case. 

\begin{example} We consider now an example for $d=2$ and $N=4$. Let $T=\begin{pmatrix} 
1 & 0 \\
0& u 
\end{pmatrix}$ with $u \geq 1$ and $F_1=\frac{1}{\sqrt{2}}\begin{pmatrix} 
0 & 1 \\
1& 0 
\end{pmatrix}, F_2=\frac{1}{\sqrt{2}}\begin{pmatrix} 
0 & i \\
-i& 0 
\end{pmatrix},F_3=\frac{1}{\sqrt{u^2+1}}\begin{pmatrix} 
u & 0 \\
0& -1 
\end{pmatrix}$. It is straightforward to check that 
$\left< F_i,F_j \right>=\delta_{ij}$ and $\operatorname{Tr}[F_iT]=0$, $i=1,2,3$. We have then
$$F:=F_1+F_2+F_3= \begin{bmatrix} 
\frac{u}{\sqrt{1+u^2}} & \frac{1}{\sqrt{2}}(1+i) \\
\frac{1}{\sqrt{2}}(1-i) & -\frac{1}{\sqrt{1+u^2}}
\end{bmatrix}.$$ 

We compare next the values of the lower and upper bounds for the largest value of the parameter $x$ giving a positive decomposition of $T$. These bounds have been obtained respectively in \eqref{eq:interval-x-PSD} and \eqref{eq:bound-a-ineq}. 

The operators $R_1$ and $R_2$ have the same eigenvalues:
\begin{equation}\lambda_{1,2}=\frac{u-1\pm \sqrt{(u-1)^2+4(13u^2+u+13)}}{6\sqrt{3}\sqrt{u^2+1}}.
\end{equation}
We denote by $\rho_1$ the largest one  and by $-\mu_1$ the smallest one. Similarly, for the matrix $R_3$, the eigenvalues are 
\begin{equation}\lambda_{3,4}=\frac{5(1-u)\pm \sqrt{25(u-1)^2+4(u^2+25u+1)}}{6\sqrt{3}\sqrt{u^2+1}}
\end{equation}
and we define $\rho_2$ to be the largest one  and $-\mu_2$ to be the smallest one. Again, for $R_4$, the eigenvalues are
\begin{equation}\lambda_{5,6}=\frac{3(u-1)\pm 3\sqrt{(u-1)^2+4(u^2+u+1)}}{6\sqrt{3}\sqrt{u^2+1}}
\end{equation}
with $\rho_3$ the largest one and $-\mu_3$ the smallest one.

We see that $\rho_2\leq \rho_1\leq \rho_3$ and $\mu_2\geq \mu_1\geq \mu_3$.
So, the bound from \eqref{eq:interval-x-PSD} reads
\begin{align*}
x_{LB} &= \frac{1}{4 \mu_2} \\
a_{LB} &= \frac{1}{16} \left(u^2+1\right)+\frac{27 \left(u^2+1\right)}{4 \left(5 u+\sqrt{u (29 u+50)+29}-5\right)^2},
\end{align*}
where we have used $t_1=\operatorname{Tr}[T]=1+u$, $t_2=\operatorname{Tr}[T^2]=1+u^2$, and $\tau = \lambda_{\min}(T) = 1$. 

When $u\geq 1$, the upper bound from \eqref{eq:bound-a-ineq} reads 
\begin{align*}
x_{UB} &= \frac{u\sqrt{1+u^2}}{4} \\
a_{UB} &= \frac{(1 + u^2)^2}{16}.
\end{align*}
We can easily checked that the two bounds are equal $a_{LB} = a_{UB}$ only when $u=1$, i.e.~when $T=I_2$, see Figure \ref{fig:example-u}.

\begin{figure}[htbp] 
\includegraphics[scale=0.8]{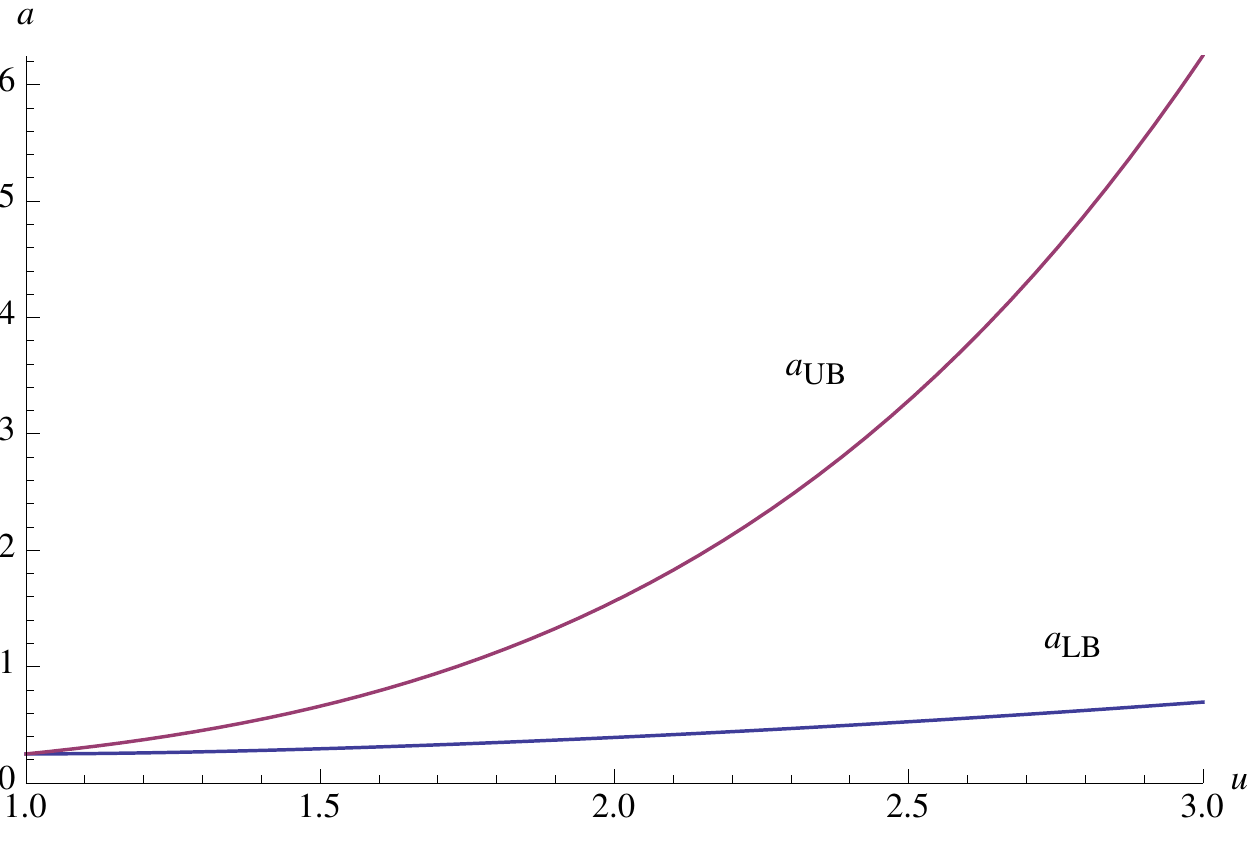}
\caption{Comparing the two bounds for the parameter $a$ of a symmetric decomposition of $T$ as a sum of $4$ positive semidefinite operators.} 
\label{fig:example-u}
\end{figure}\end{example}

With the help of a computer\footnote{see the supplementary material for the arXiv preprint.}, we have found that the actual largest value of the parameter $a$ for which there exist symmetric positive decompositions of $T$ is 
$$a_{opt} = \frac{1}{16}\left(1+u^2+\frac{27 \left(-5 \sqrt{u^2+1} u^2+5 \sqrt{u^2+1}+\left(u^2+1\right) \sqrt{u (25 u+4)+25}\right)^2}{4 (u (u+25)+1)^2}\right).$$
Interestingly, it turns out that this value is very close to the lower bound in \eqref{eq:interval-x-PSD}:
$$0 \leq a_{opt}-a_{LB} \leq \frac{27}{800} \left(125 \sqrt{29}-673\right) \approx 0.00491403.$$

\section{Dual symmetric decompositions}
\label{4}
In the following we consider the \emph{dual family} associated to a given non-degenerate symmetric family $\mathcal{E}=\{E_i\}_{i=1}^N$ and we show that, after rescaling, it also gives a symmetric decomposition of $T= \sum_i E_i$. Recall  that the dual family  $\tilde{\mathcal{E}}=\{\tilde{E}_i\}_{i=1}^N$ is another set of $N$ self-adjoint operators, having the same span as $\mathcal E$, and the additional property $\operatorname{Tr}[E_i\tilde{E}_j]=\delta_{ij}$, $\forall i,j=1,\ldots, N$. It is easy to check that the operators of the dual family $\tilde{E}_i$ are given by $\tilde{E}_i=\sum_{j=1}^{N} (G^{-1})_{ij}E_j$, where $G \in M_N^{sa}(\mathbb R)$ is the Gram matrix of $\mathcal{E}$, i.e.~$G_{ij} = \operatorname{Tr}[E_i E_j]$. Since we assume that the family $E$ is symmetric with parameters $a$, $b$, we have $G=(a-b)I_N+bJ_N$, where $J_N$ is the matrix with all entries equal to $1$. Moreover, we have assumed that $\mathcal E$ is non-degenerate, so $a > |b|$; it follows that 
$$G^{-1}=\frac{1}{a-b}\left[I_N-\frac{b}{a+b(N-1)}J_N\right].$$
Consequently, the dual family  $\tilde{\mathcal{E}}=\{\tilde{E}_i\}_{i=1}^N$ is given by 
\begin{equation}
\tilde{E}_i=\frac{1}{a-b}\left[E_i-\frac{b}{a+b(N-1)}\sum_{j=1}^N E_j\right], \qquad \forall 1 \leq i \leq N,
\end{equation}
and it is also a symmetric family of parameters
\begin{align}
\label{a-b-tilde1} \tilde{a}=&\frac{a+b(N-2)}{(a-b)(a+b(N-1))} \\
\label{a-b-tilde2} \tilde{b}=&-\frac{b}{(a-b)(a+b(N-1))}
\end{align}

It is of interest to study the properties of the dual family $\tilde{\mathcal{E}}$ in the case when the family $\mathcal{E}=\{E_i\}_{i=1}^N$ is a non-degenerate symmetric  decomposition of an self-adjoint operator $T$.   
\begin{proposition}
Let $T$ be a self-adjoint operator and $\mathcal{E}=\{E_i\}, i=1,\dots N $ a non-degenerate symmetric decomposition of $T$. Then, the dual family $\{\tilde{E_i}\}$ is given by
\begin{equation}\label{eq:dual-family}
\tilde{E}_i=\frac{N(N-1)}{aN^2-t_2}\left[ E_i-\frac{t_2-aN}{(N-1)t_2} T  \right].
\end{equation}
and has parameters
\begin{align*}
\tilde{a} &= \frac{N}{t_2}\cdot\frac{(N-2)t_2+aN}{aN^2-t_2}\\ 
\tilde{b} &=- \frac{N}{t_2}\cdot\frac{t_2-aN}{aN^2-t_2}
\end{align*}

\end{proposition}
\begin{proof} The results follow from direct computation using Proposition \ref{prop:relation-a-b-T} and equations \eqref{a-b-tilde1}--\eqref{a-b-tilde2}.
\end{proof}

Note that, when $T=I_d$ and $N=d^2$, the dual family given by \eqref{eq:dual-family} corresponds to the dual basis associated to (general) SIC--POVMs \cite[Section 2]{gour}.
\begin{remark}
 
The \emph{normalized dual family} $\hat{\mathcal{E}}=\{\hat{E}_i\}_{i=1}^N$, given by
$\hat{E}_i=\frac{t_2}{N}\tilde{E}_i$,
forms a new symmetric decomposition of $T$ of parameters
\begin{align*}
\hat a & =\frac{t_2}{N} \cdot  \frac{(N-2)t_2+aN}{aN^2-t_2}\\ 
\hat{b} &=- \frac{t_2}{N} \cdot \frac{t_2-aN}{aN^2-t_2}
\end{align*}
\end{remark}
 
\begin{remark}
Note that map which associates to a symmetric family $\mathcal E$ its (normalized) dual family $\hat{\mathcal E}$ is an involution; in particular $\hat{\hat a} = a$. 
\end{remark}

Since both families $\mathcal E$ and $ \hat{\mathcal E}$ are symmetric decompositions of $T$, it is of interest to relate the decomposition of the operator $T$ by the family $ \hat{\mathcal{E}} $ using a similar procedure as described for the family $\mathcal{E}=\{E_i\}$ in Proposition \ref{prop:general-symmetric-decompositions}. By straightforward computations, it is possible to show that in this case, starting from an orthonormal basis $ \{F_i\}$, $i=1,\ldots, N-1$, with the \emph{same} operators $R_i$, $i=1,\ldots, N$ as given by \eqref {eq:R-i} and  \eqref{eq:R-N}, we get that   $\hat{E}_i=\frac{1}{N}T+{\hat{x}}R_i$, where 
$$\hat x :=\frac{(N-1)t_2}{N^2a-t_2}x,$$ 
for any positive real $x$ ($x=0$ is not allowed here, since we have assumed the primal family $\mathcal E$ to be non-degenerate). Using  the expression of the symmetry parameter $a$ as given by \eqref{eq:a-general-decomposition}, it follows that 
$$\hat x = \frac{(N-1)t_2}{N^2}\frac{1}{x}.$$
As before, one may use Proposition  \ref{prop:general-symmetric-decompositions} to obtain a sufficient condition for positivity of the decomposition, see \eqref{eq:interval-x-PSD}. 

\section{Decompositions and Welch--type inequalities}
\label{5}
The following result is known in the literature as the \emph{simplex bound}. The idea, originating in \cite[Corollary 5.2]{chs}, is that among $N$ subspaces of fixed dimension in $\mathbb C^d$, there must be at least a pair with ``small'' principal angles. This result has been generalized to subspaces with weights in \cite[Theorem 3.4]{bachoc}, and then to arbitrary positive semidefinite operators with fixed trace in \cite[Proposition 4.1]{cahill}. In the result below, we slightly generalize this last result, by removing the fixed trace condition. The equality case has been recognized to play an important role, characterizing tight fusion frames, see \cite{welch}, \cite[Theorem 4.3]{kpcl}.

\begin{proposition}\label{prop:Welch}
Consider $\mathcal{E}=\{E_i\}_{i=1}^N$ a family of self-adjoint
operators on $\mathcal{H} \simeq \mathbb C^d$. Then, we have
\begin{equation}\label{eq:simplex-bound}
\max_{i\neq j}\langle E_i,E_j \rangle\geq \frac{d^{-1}(\sum_{j=1}^N\mathrm{Tr}[E_j])^2-\sum_{j=1}^N\langle E_j,E_j \rangle}{N(N-1)},
\end{equation}
with equality iff $\mathcal{E}$ is equiangular and $\sum_{i=1}^N E_i = \left( d^{-1}\sum_{i=1}^N \mathrm{Tr}[E_i]\right) I_d$.
\end{proposition}

\begin{proof}
We have
\begin{align}
\label{eq:sum-different-indices-1} \sum_{\substack{i,j=1 \\ i\neq j}}^N \langle E_i,E_j \rangle &= \left\langle\sum_{i=1}^N E_i,\sum_{i=1}^N E_i \right\rangle - \sum_{i=1}^N \langle E_i, E_i \rangle\\
\label{eq:sum-different-indices-2}  &=\mathrm{Tr}\left[\sum_{i=1}^N E_i \right]^2-\sum_{i=1}^N \langle E_i, E_i \rangle\geq \frac{1}{d}\left(\sum_{i=1}^N \mathrm{Tr} [E_i]\right)^2-\sum_{i=1}^N \langle E_i, E_i \rangle
\end{align}
Using the fact that the maximum of all the terms in the LHS of the equation above is larger than the average term, \eqref{eq:simplex-bound} follows.
We have equality iff all the terms are equal, and thus the family $\mathcal E$ is equiangular. In this case, from the equality case in the Cauchy-Schwarz inequality, we have $\sum_{j=1}^N E_j=\lambda I_d$. But $\sum_{i=1}^N \mathrm{Tr} [E_i]=\lambda \mathrm{Tr} [I_d]$, and thus $\lambda = d^{-1}\sum_{i=1}^N \mathrm{Tr} [E_i]$.
\end{proof}

\begin{remark} If $\sum_{i=1}^N \mathrm{Tr} [E_i]=d$ and $\lambda=1$, we recover the statement of \cite[Proposition 4.1]{cahill}.
\end{remark}

In the following we give some extensions and generalizations of a result from \cite{wol}. We use an idea from \cite{bachoc}, which requires to introduce scalar weights $v_i$.  

\begin{proposition}\label{prop:Welch-weights}
Consider $\mathcal{E}=\{E_i\}_{i=1}^N$ a family of positive operators on $\mathcal{H} \simeq \mathbb C^d$, so that $\mathrm{Tr} [E_i^2]=1$, for all $1 \leq i \leq N$. Then, for any positive weights $v_1,\ldots, v_N >0$, we have
\begin{eqnarray}
\sum_{\substack{i,j=1 \\ i\neq j}}^N v_iv_j \langle E_i,E_j \rangle \geq d^{-1} \left(\sum_{i=1}^N v_i \right)^2- \sum_{i=1}^N v_i^2,
\end{eqnarray}
with equality iff the $E_i$ are rank-one projections and  $\sum_{i=1}^N v_i E_i= \left( d^{-1}\sum_{i=1}^N v_i \right)I_d$.
\end{proposition}
\begin{proof}
We apply Proposition \ref{prop:Welch} to the operators $F_j = v_j E_j$; using \eqref{eq:sum-different-indices-1}-\eqref{eq:sum-different-indices-2}, we get
$$\sum_{\substack{i,j=1 \\ i\neq j}}^N v_iv_j\langle E_i,E_j \rangle \geq d^{-1} \left(\sum_{i=1}^N v_i \operatorname{Tr} E_i  \right)^2- \sum_{i=1}^N v_i^2,$$
To conclude, we use $1=\operatorname{Tr}[E_i^2]\leq \operatorname{Tr} E_i$, with equality iff each $E_i$ is a rank-one projector.
\end{proof}

\begin{remark}
In order for the inequality in the statement to be non trivial, the \emph{weight coefficient}
$$[v]:=\frac{\left(\sum_{i=1}^N v_i\right)^2}{\sum_{i=1}^N v_i^2}$$
must satisfy $[v] \geq d$. Note that in general, $[v]$ lies in the interval $[1,N]$. 
\end{remark}

\begin{proposition}\label{prop:Welch-weights-Holder}
Let $\mathcal{E}=\{E_i\}_{i=1}^N$ a family of positive semidefinite operators on $\mathcal{H} \simeq \mathbb C^d$, so that $\mathrm{Tr} [E_i^2]=1$, for all $1 \leq i \leq N$. Then, for any $1 < p < \infty$ and any positive weights $v_1,\ldots, v_N >0$ such that $[v] \geq d$, we have
\begin{equation}
\sum_{1 \leq i \neq j \leq N} v_iv_j \langle E_i, E_j\rangle^p \geq \frac{\left[d^{-1}\left(\sum_i v_i\right)^2- \sum_i  v_i^2 \right]^p}{ \left[\sum_{i\neq j} v_i v_j\right]^{p-1}},
\end{equation}
with equality iff the $E_i$ are equiangular rank-one projections and  $\sum_{i=1}^N v_i E_i= \left( d^{-1}\sum_{i=1}^N v_i \right)I_d$.
\end{proposition}

\begin{proof}
From H\"older's inequality it follows that 
\begin{align*}
\sum_{i\neq j} v_iv_j \langle E_i, E_j\rangle &=  \sum_{i\neq j} v_i^{\frac 1 p} v_j^{\frac 1 p} \langle E_i, E_j\rangle \cdot v_i^{\frac 1 q} v_j^{\frac 1 q}\\
&\leq 
\left[ \sum_{i\neq j} v_i v_j \langle E_i, E_j\rangle^p \right]^{\frac 1 p} \cdot 
\left[ \sum_{i\neq j} v_i v_j \right]^{\frac 1 q},
\end{align*}
where $p$ and $q$ are conjugate exponents $p^{-1} + q^{-1}=1$. Therefore, using Proposition \ref{prop:Welch-weights}, we have
\begin{align*}
\sum_{i\neq j} v_iv_j \langle E_i,E_j \rangle^p &\geq \frac{\left[ \sum_{i\neq j} v_iv_j \langle E_i,E_j \rangle \right]^p}{\left[ \sum_{i\neq j} v_i v_j \right]^{p-1}}\\\nonumber 
&\geq\frac{\left[d^{-1}\left(\sum_i v_i\right)^2- \sum_i  v_i^2 \right]^p}{ \left[\sum_{i\neq j} v_i v_j\right]^{p-1}}.
\end{align*}
H\"older's inequality is saturated iff $v_i v_j \langle E_i, E_j\rangle^p = \lambda v_iv_j$ for all $i \neq j$, i.e.~iff the family $\mathcal E$ is equiangular.

\end{proof}

\begin{remark} With the choice  $v_1 = \cdots v_N = 1$ and $p=2$, Proposition \ref{prop:Welch-weights-Holder} gives the result from \cite[Proposition 2.7 (SIC--POVMs)]{wol}:
$$\sum_{j\neq k}\langle E_j,E_k \rangle^2 \geq \frac{N(N-d)^2}{(N-1)d^2}.$$
\end{remark} 

Either from Proposition \ref{prop:Welch-weights} or from Proposition \ref{prop:Welch-weights-Holder}, one obtains the following weight generalization of the simplex bound. 

\begin{corollary}
Let $\mathcal{E}=\{E_i\}_{i=1}^N$ a family of positive semidefinite operators on $\mathcal{H} \simeq \mathbb C^d$, so that $\mathrm{Tr} [E_i^2]=1$, for all $1 \leq i \leq N$. Then, for any $1 < p < \infty$ and any positive weights $v_1,\ldots, v_N >0$ such that $[v] \geq d$, we have
\begin{equation}\label{eq:min-angle-weights}
\max_{1 \leq i \neq j \leq N} \langle E_i, E_j \rangle \geq \frac{N-d}{d(N-1)} \geq \frac{d^{-1}\left(\sum_i v_i\right)^2- \sum_i  v_i^2}{\sum_{i\neq j} v_i v_j},
\end{equation}
with equality iff the $E_i$ are equiangular rank-one projections and  $\sum_{i=1}^N v_i E_i= \left( d^{-1}\sum_{i=1}^N v_i \right)I_d$.
\end{corollary}
\begin{proof}
Note that the left hand side of \eqref{eq:min-angle-weights} does not depend on the weights $v$, so we just need to show that the right hand side is maximal when all the weights are equal. Using the homogeneity of the expression, we can assume $\sum_i v_i = 1$, i.e.~$v$ is a probability vector. Replacing two components $v_i > v_j$ of $v$ with $v_i - \varepsilon$ and respectively $v_j + \varepsilon$, for $\varepsilon>0$ small enough, we see that the bound increases, so the maximum must be achieved by ``flat'' weights $v$ (see \cite[Theorem II.1.10]{bha} for the related concept of majorization).
\end{proof}

\end{document}